\documentclass[11pt]{amsart}
\usepackage{amsmath}
\usepackage{amssymb}
\usepackage{amsthm}
\usepackage{amscd}
\usepackage{enumerate}
\usepackage{verbatim}
\usepackage[all]{xy}
\usepackage{mathrsfs}

\theoremstyle{plain}
\newtheorem{thm}{Theorem}[section]
\newtheorem{prop}[thm]{Proposition}
\newtheorem{lem}[thm]{Lemma}
\newtheorem{cor}[thm]{Corollary}

\theoremstyle{definition}

\newcommand{\Hom}{\mathrm{Hom}}

\newcommand{\Lie}{\mathrm{Lie}}

\newcommand{\Ker}{\mathrm{Ker}}
\newcommand{\Coker}{\mathrm{Coker}}
\newcommand{\Img}{\mathrm{Im}}
\newcommand{\prjt}{\mathrm{pr}}
\newcommand{\Fil}{\mathrm{Fil}}
\newcommand{\Spf}{\mathrm{Spf}}

\newcommand{\Spec}{\mathrm{Spec}}

\newcommand{\Gal}{\mathrm{Gal}}

\newcommand{\Frac}{\mathrm{Frac}}

\newcommand{\Span}{\mathrm{Span}}

\newcommand{\Ha}{\mathrm{Ha}}
\newcommand{\Hdg}{\mathrm{Hdg}}

\newcommand{\HT}{\mathrm{HT}}

\newcommand{\sExt}{\mathscr{E}\!{\it xt}}

\newcommand{\Kbar}{\bar{K}}

\newcommand{\okbar}{\mathcal{O}_{\bar{K}}}

\newcommand{\okey}{\mathcal{O}_K}

\newcommand{\oel}{\mathcal{O}_L}

\newcommand{\cA}{\mathcal{A}}

\newcommand{\cC}{\mathcal{C}}
\newcommand{\cD}{\mathcal{D}}
\newcommand{\cE}{\mathcal{E}}
\newcommand{\cF}{\mathcal{F}}
\newcommand{\cG}{\mathcal{G}}
\newcommand{\cH}{\mathcal{H}}

\newcommand{\cK}{\mathcal{K}}

\newcommand{\cM}{\mathcal{M}}
\newcommand{\cN}{\mathcal{N}}
\newcommand{\cO}{\mathcal{O}}

\newcommand{\cU}{\mathcal{U}}

\newcommand{\Ga}{\mathbb{G}_\mathrm{a}}

\newcommand{\TSG}{T^{*}_{\SG}}

\newcommand{\okp}{\mathcal{O}_{K_1}}

\newcommand{\tokbar}{\tilde{\mathcal{O}}_{\Kbar}}
\newcommand{\tokey}{\tilde{\mathcal{O}}_{K}}

\newcommand{\PD}{\mathrm{DP}}

\newcommand{\CRYS}{\mathrm{CRYS}}

\newcommand{\Gr}{\mathrm{Gr}}
\newcommand{\Mod}{\mathrm{Mod}}

\newcommand{\SG}{\mathfrak{S}}
\newcommand{\SGe}{\mathfrak{E}}
\newcommand{\SGf}{\mathfrak{F}}
\newcommand{\SGm}{\mathfrak{M}}
\newcommand{\SGn}{\mathfrak{N}}
\newcommand{\SGl}{\mathfrak{L}}

\newcommand{\ModSGf}{\mathrm{Mod}_{/\mathfrak{S}_1}^{1,\varphi}}

\newcommand{\ModSGffr}{\mathrm{Mod}_{/\mathfrak{S}}^{1,\varphi}}

\newcommand{\ModSf}{\mathrm{Mod}_{/S_1}^{1,\varphi}}

\newcommand{\ModSffr}{\mathrm{Mod}_{/S}^{1,\varphi}}

\newcommand{\upi}{\underline{\pi}}

\newcommand{\bC}{\mathbb{C}}
\newcommand{\bD}{\mathbb{D}}

\newcommand{\bZ}{\mathbb{Z}}
\newcommand{\bQ}{\mathbb{Q}}
\newcommand{\bF}{\mathbb{F}}

\newcommand{\sH}{\mathscr{H}}
\newcommand{\sI}{\mathscr{I}}
\newcommand{\sJ}{\mathscr{J}}

\newcommand{\sS}{\mathscr{S}}

\newcommand{\frB}{\mathfrak{B}}
\newcommand{\frG}{\mathfrak{G}}
\newcommand{\frI}{\mathfrak{I}}

\newcommand{\frP}{\mathfrak{P}}

\newcommand{\frX}{\mathfrak{X}}

\newcommand{\frZ}{\mathfrak{Z}}

\makeatletter
\renewcommand{\p@enumii}{}
\makeatother

\begin{document}

\title[Canonical subgroups via Breuil-Kisin modules]{Canonical subgroups via Breuil-Kisin modules}
\author{Shin Hattori}
\date{\today}
\email{shin-h@math.kyushu-u.ac.jp}
\address{Faculty of Mathematics, Kyushu University}
\thanks{Supported by Grant-in-Aid for Young Scientists B-21740023.}

\begin{abstract}
Let $p>2$ be a rational prime and $K/\bQ_p$ be an extension of complete discrete valuation fields. Let $\cG$ be a truncated Barsotti-Tate group of level $n$, height $h$ and dimension $d$ over $\okey$ with $0<d<h$. In this paper, we show that if the Hodge height of $\cG$ is less than $1/(p^{n-2}(p+1))$, then there exists a finite flat closed subgroup scheme of $\cG$ of order $p^{nd}$ over $\okey$ with standard properties as the canonical subgroup.
\end{abstract}

\maketitle

\section{Introduction}\label{intro}

Let $p$ be a rational prime and $K/\bQ_p$ be an extension of complete discrete valuation fields. Let $k$ be the residue field of $K$ and $e=e(K)$ be the absolute ramification index of $K$. Let $\okey$ denote the ring of integers of $K$ and put $\tokey=\okey/p\okey$. We fix an algebraic closure $\Kbar$ of $K$ and let $v_p$ be the valuation of $K$ which is normalized as $v_p(p)=1$ and extended to $\Kbar$. For a non-negative rational number $i$, we put $m_{\Kbar}^{\geq i}=\{a\in \okbar\mid v_p(a)\geq i\}$ and similarly for $m_K^{\geq i}$.

Let $E$ be an elliptic curve over $\okey$. By fixing a formal parameter of the formal completion of $E$ along the zero section, we identify the group of $p$-torsion points $E[p](\okbar)$ as a subset of $m_{\Kbar}$. If the group $E[p](\okbar)$ has a subgroup $C$ of order $p$ whose elements have valuations greater than those of the elements of $E[p](\okbar)\setminus C$, then the subgroup $C$ is called the (level one) canonical subgroup of $E$. It was shown that the canonical subgroup of $E$ exists if the Hodge height of $E$, namely the $p$-adic valuation of the Hasse invariant of $E$, is less than $p/(p+1)$, and based on this result, Katz studied a spectral theory of the $U_p$ operator for overconvergent elliptic modular forms (\cite{Ka}).

For a similar investigation of $p$-adic modular forms for general reductive algebraic groups, we need a generalization of the existence theorem of the canonical subgroup to higher dimensional abelian schemes over $\okey$. Such a generalization was first obtained by Abbes and Mokrane via a calculation of $p$-adic vanishing cycles of abelian schemes (\cite{AM}). Namely, for an abelian scheme $A$ over $\okey$ with relative dimension $g$, they found a subgroup of $A[p](\okbar)$ of order $p^g$ in the upper ramification filtration $\{A[p]^{j+}(\okbar)\}$ if the Hodge height of $A$ (\cite[Subsection 1.2]{AM}) is less than an explicit bound. Their work was followed by many improvements and generalizations with various methods, such as \cite{AG}, \cite{Co}, \cite{GK1}, \cite{GK2}, \cite{KL}, \cite{Ra}, \cite{Ti} and especially \cite{Fa} and \cite{Ti_HN}.

In this paper, we prove an existence theorem of level $n$ canonical subgroups for $p>2$, using a classification theory of finite flat (commutative) group schemes due to Breuil and Kisin (\cite{Br}, \cite{Br_AZ}, \cite{Ki_Fcrys}, \cite{Ki_BT}), and a ramification-theoretic technique developed by the author (\cite{Ha4}). 

Before stating the main theorem, we fix some notation. For a finite flat group scheme $\cG$ over $\okey$ and its module of invariant differentials $\omega_{\cG}$ over $\okey$, write $\omega_{\cG}\simeq \oplus_i \okey/(a_i)$ with some $a_i\in \okey$ and put $\deg(\cG)=\sum_iv_p(a_i)$. We define the Hodge-Tate map
\[
\HT_i: \cG(\okbar)\to \omega_{\cG^\vee}\otimes \okbar/m_{\Kbar}^{\geq i}
\]
by $x\mapsto (x^\vee)^{*}(dt/t)$, where $x^\vee: \cG^\vee\times \Spec(\okbar)\to \mu_p$ is the dual map of $x\in\cG(\okbar)$. Let $\cG^j$ and $\cG^{j+}$ ({\it resp.} $\cG_i$ and $\cG_{i+}$) denote the upper ({\it resp.} lower) ramification subgroup schemes of $\cG$ (see \cite{Ha4}). Here we normalize the indices of the filtrations by multiplying the usual ones by $1/e(K)$, so that these ramification subgroup schemes are compatible with any base change of complete discrete valuation rings. In particular, $\cG_i$ is defined by
\[
\cG_i(\okbar)=\Ker(\cG(\okbar)\to \cG(\okbar/m_{\Kbar}^{\geq i})).
\]
For a truncated Barsotti-Tate group $\cG$ of level $n$, height $h$ and dimension $d<h$ over $\okey$, let us consider its Cartier dual $\cG^\vee$ and the $p$-torsion subgroup scheme $\cG^\vee[p]$. Then the Lie algebra $\Lie(\cG^\vee[p]\times\Spec(\tokey))$ is a free $\tokey$-module of rank $h-d$. We define the Hodge height $\Hdg(\cG)$ of $\cG$ to be the truncated valuation $v_p(\det(V_{\cG^\vee[p]}))\in [0,1]$ of the determinant of the action of the Verschiebung $V_{\cG^\vee[p]}$ on this $\tokey$-module. Note the equality $\Hdg(\cG)=\Hdg(\cG^\vee)$ (see for example \cite[Proposition 2]{Fa}, where the Hodge height of $\cG$ is denoted by $\Ha(\cG)$ and referred as the Hasse invariant of $\cG$). Then the main theorem of this paper is the following.

\begin{thm}\label{main}
Let $p>2$ be a rational prime and $K/\bQ_p$ be an extension of complete discrete valuation fields. Let $\cG$ be a truncated Barsotti-Tate group of level $n$, height $h$ and dimension $d$ over $\okey$ with $0<d<h$ and Hodge height $w=\Hdg(\cG)$. 
\begin{enumerate}
\item If $w<1/(p^{n-2}(p+1))$, then there exists a finite flat closed subgroup scheme $\cC_n$ of $\cG$ of order $p^{nd}$ over $\okey$, which we call  the level $n$ canonical subgroup of $\cG$, such that $\cC_n\times \Spec(\okey/m_K^{\geq 1-p^{n-1}w})$ coincides with the kernel of the $n$-th iterated Frobenius homomorphism $F^n$ of $\cG\times \Spec(\okey/m_K^{\geq 1-p^{n-1}w})$. Moreover, the group scheme $\cC_n$ has the following properties:
\begin{enumerate}[(a)]
\item\label{main-deg} $\deg(\cG/\cC_n)=w(p^n-1)/(p-1)$.
\item\label{main-iso} Put $\cC_n'$ to be the level $n$ canonical subgroup of $\cG^\vee$. Then we have the equality of subgroup schemes $\cC_n'=(\cG/\cC_n)^\vee$, or equivalently $\cC_n(\okbar)=\cC_n'(\okbar)^\bot$, where $\bot$ means the orthogonal subgroup with respect to the Cartier pairing.
\item\label{main-ram1} If $n=1$, then $\cC_1=\cG_{(1-w)/(p-1)}=\cG^{pw/(p-1)+}$.
\end{enumerate}

\item If $w<(p-1)/(p^n-1)$, then the subgroup scheme $\cC_n$ also satisfies the following: 
\begin{enumerate}[(a)]
\setcounter{enumii}{3}
\item\label{main-free}   the group $\cC_n(\okbar)$ is isomorphic to $(\bZ/p^n\bZ)^d$.
\item\label{main-sub} The scheme-theoretic closure of $\cC_n(\okbar)[p^i]$ in $\cC_n$ coincides with the subgroup scheme $\cC_i$ of $\cG[p^i]$ for $1\leq i\leq n-1$.
\end{enumerate}

\item\label{main-HT} If $w<(p-1)/p^n$, then the subgroup  $\cC_n(\okbar)$ coincides with the kernel of the Hodge-Tate map $\HT_{n-w(p^n-1)/(p-1)}$.

\item\label{main-ram} If $w<1/(2p^{n-1})$, then the subgroup scheme $\cC_n$ coincides with the upper ramification subgroup scheme $\cG^{j+}$ for 
\[
p w(p^n-1)/(p-1)^2 \leq j <p(1-w)/(p-1).
\]
\end{enumerate}
\end{thm}
Moreover, we show that such $\cC_n$ is unique if $w<p(p-1)/(p^{n+1}-1)$ (Proposition \ref{uniqueFrn}). Since the upper ramification subgroups can be patched into a family (see Lemma \ref{familyup}), we also have the following corollary.

\begin{cor}\label{cansubfamily}
Let $K/\bQ_p$ be an extension of complete discrete valuation fields and $j$ be a positive rational number. Let $\frX$ be an admissible formal scheme over $\Spf(\okey)$ which is quasi-compact and $\frG$ be a truncated Barsotti-Tate group of level $n$ over $\frX$ of constant height $h$ and dimension $d$ with $0<d<h$. We let $X$ and $G$ denote the Raynaud generic fibers of the formal schemes $\frX$ and $\frG$, respectively. For a finite extension $L/K$ and $x\in X(L)$, we put $\frG_x=\frG\times_{\frX,x}\Spf(\oel)$, where we let $x$ also denote the map $\Spf(\oel)\to \frX$ obtained from $x$ by taking the scheme-theoretic closure and the normalization. For a non-negative rational number $r$, let $X(r)$ be the admissible open subset of $X$ defined by 
\[
X(r)(\Kbar)=\{x\in X(\Kbar)\mid \Hdg(\frG_x)<r\}.
\]
Put $r_1=p/(p+1)$ and $r_n=1/(2p^{n-1})$ for $n\geq 2$. 

Suppose $p>2$. Then there exists an admissible open subgroup $C_n$ of $G|_{X(r_n)}$ such that, etale locally on $X(r_n)$, the rigid-analytic group $C_n$ is isomorphic to the constant group $(\bZ/p^n\bZ)^d$ and, for any finite extension $L/K$ and $x\in X(L)$, the fiber $(C_n)_x$ coincides with the generic fiber of the level $n$ canonical subgroup of $\frG_x$.
\end{cor}

Note that for a smaller range of $w$, Theorem \ref{main} is also proved in \cite{Fa} and \cite{Ti_HN}. The key idea of our approach is, firstly, to lift the conjugate Hodge filtration of $\cG\times \Spec(\tokey)$ to the Breuil-Kisin module associated to $\cG$. By an induction, it suffices to consider the case of $n=1$. We may assume that the residue field $k$ is perfect and that the group scheme $\cG$ is associated to a $\varphi$-module $\SGm$ over the formal power series ring $k[[u]]$ via the Breuil-Kisin classification (see Section \ref{ReviewBK}). The Lie algebra $\Lie(\cG^\vee)$ is naturally considered as a $\varphi$-stable direct summand of the $\varphi$-module $\SGm/u^e\SGm$. Then we show that the reduction modulo $u^{e(1-w)}$ of this direct summand lifts uniquely to a $\varphi$-stable direct summand of $\SGm$. Our canonical subgroup is defined to be the finite flat closed subgroup scheme of $\cG$ associated to the quotient of $\SGm$ by this direct summand. Its properties stated in Theorem \ref{main} follow easily from the construction, basically except the uniqueness and the coincidence with ramification subgroup schemes. The second key idea is to switch to a complete discrete valuation field of equal characteristic: From the $\varphi$-module $\SGm$, we can also construct a finite flat group scheme $\cH(\SGm)$ over $k[[u]]$ (\cite{SGA3-7A}). Then, by the main theorem of \cite{Ha4}, the ramification subgroups of $\cG$ and $\cH(\SGm)$ are naturally isomorphic to each other. Moreover, it is also proved that reductions of $\cG$ and $\cH(\SGm)$ are isomorphic as pointed schemes (\cite[Corollary 4.6]{Ha4}). These reduce proofs of the remaining properties of our canonical subgroup to an elementary calculation on the side of equal characteristic, which we can easily accomplish.
\\

\noindent
{\bf Acknowledgments.} The author would like to thank the anonymous referee for many helpful comments to improve the paper.



\section{Review of the Breuil-Kisin classification of finite flat group schemes and their ramification theory}\label{ReviewBK}

Let the notation be as in the previous section and suppose that the residue field $k$ of $K$ is perfect of characteristic $p>2$. Let $W=W(k)$ be the Witt ring of $k$ and $\varphi$ be the Frobenius endomorphism of $W$. Let us fix once and for all a uniformizer $\pi$ of $K$ and a system $\{\pi_n\}_{n\in \bZ_{\geq 0}}$ of its $p$-power roots in $\Kbar$ with $\pi_0=\pi$ and $\pi_{n+1}^p=\pi_n$. Put $K_n=K(\pi_n)$ and $K_\infty=\cup K_n$. Let $E(u)\in W[u]$ be the Eisenstein polynomial of $\pi$ over $W$ and put $c_0=p^{-1}E(0)$. In this section, we briefly recall a classification theory of Breuil (\cite{Br}, \cite{Br_AZ}) and Kisin (\cite{Ki_Fcrys}, \cite{Ki_BT}) of finite flat group schemes over $\okey$ and their ramification theory, while we concentrate mainly on the $p$-torsion case.

\subsection{Breuil and Kisin modules}

Put $\SG=W[[u]]$ and $\SG_1=k[[u]]$. The $\varphi$-semilinear continuous ring endomorphisms of these rings defined by $u\mapsto u^p$ are also denoted by $\varphi$. Then a Kisin module over $\SG$ (of $E$-height $\leq 1$) is an $\SG$-module $\SGm$ endowed with a $\varphi$-semilinear map $\varphi_\SGm:\SGm \to \SGm$ such that the cokernel of the map $1\otimes \varphi_\SGm: \SG\otimes_{\varphi,\SG}\SGm\to \SGm$ is killed by $E(u)$. We write $\varphi_\SGm$ also as $\varphi$ if there is no risk of confusion. A morphism of Kisin modules is an $\SG$-linear map which is compatible with $\varphi$'s of the source and the target. Then the Kisin modules form a category and this category has an obvious notion of exact sequences. We let $\ModSGf$  ({\it resp.} $\ModSGffr$) denote its full subcategory consisting of the objects whose underlying $\SG$-modules are free of finite rank over $\SG_1$ ({\it resp.} $\SG$). For a $k[[u]]$-algebra $B$, we write $\varphi$ also for the $p$-th power Frobenius endomorphism of $B$. Then we let $\mathrm{Mod}_{/B}^{1,\varphi}$ denote the category of locally free $B$-modules $\SGm$ of finite rank endowed with a $\varphi$-semilinear map $\varphi_\SGm:\SGm\to \SGm$ such that the cokernel of the map $1\otimes \varphi_\SGm:B\otimes_{\varphi,B}\SGm\to \SGm$ is killed by $E(u)$.

We also have categories $\ModSf$ and $\ModSffr$ of Breuil modules defined as follows. Let $S$ be the $p$-adic completion of the divided power envelope $W[u]^\PD$ of $W[u]$ with respect to the ideal $(E(u))$ and the compatibility condition with the canonical divided power structure on $pW$. The ring $S$ has a natural filtration $\Fil^iS$ defined as the closure in $S$ of the ideal generated by $E(u)^j/j!$ for integers $j\geq i$. The $\varphi$-semilinear continuous ring homomorphism $S\to S$ defined by $u\mapsto u^p$ is also denoted by $\varphi$. We have $\varphi(\Fil^1 S)\subseteq p S$ and put $\varphi_1=p^{-1}\varphi|_{\Fil^1 S}$. This filtration and the map $\varphi_1$ induce a similar structure on the ring $S_n=S/p^nS$. Note that we have $\varphi(\Fil^1 S_1)=0$. Put $c=\varphi_1(E(u))\in S^\times$. Then we let ${}'\Mod_{/S}^{1,\varphi}$ denote the category of $S$-modules $\cM$ endowed with an $S$-submodule $\Fil^1\cM$ containing $(\Fil^1S)\cM$ and a $\varphi$-semilinear map $\varphi_{1,\cM}:\Fil^1\cM\to \cM$ satisfying $\varphi_{1,\cM}(s_1 m)=c^{-1}\varphi_1(s_1)\varphi_{1,\cM}(E(u)m)$ for any $s_1\in \Fil^1S$ and $m\in\cM$. A morphism of this category is defined to be a homomorphism of $S$-modules compatible with $\Fil^1$'s and $\varphi_1$'s. Note that this category also has an obvious notion of exact sequences. We drop the subscript $\cM$ of $\varphi_{1,\cM}$ if no confusion may occur. We let $\ModSf$ ({\it resp.} $\ModSffr$) denote the full subcategory of ${}'\Mod_{/S}^{1,\varphi}$ consisting of the objects $\cM$ such that $\cM$ is free of finite rank over $S_1$ ({\it resp.} $\cM$ is free of finite rank over $S$, $\cM/\Fil^1\cM$ is $p$-torsion free) and the image $\varphi_{1,\cM}(\Fil^1\cM)$ generates the $S$-module $\cM$. 

The categories of Breuil and Kisin modules are in fact equivalent. By the natural map $\SG\to S$, we consider the ring $S$ as an $\SG$-algebra. We define an exact functor $\cM_\SG(-):\ModSGf\to \ModSf$ by putting $\cM_\SG(\SGm)=S\otimes_{\varphi,\SG}\SGm$ with
\begin{align*}
&\Fil^1\cM_\SG(\SGm)=\Ker(S\otimes_{\varphi,\SG}\SGm\overset{1\otimes\varphi}{\to}(S_1/\Fil^1S_1)\otimes_{\SG}\SGm),\\
&\varphi_1:\Fil^1\cM_\SG(\SGm)\overset{1\otimes \varphi}{\to}\Fil^1S_1\otimes_{\SG}\SGm\overset{\varphi_1\otimes 1}{\to} S_1\otimes_{\varphi,\SG_1}\SGm=\cM_\SG(\SGm).
\end{align*}
Then the functor $\cM_\SG(-)$ is an equivalence of categories (\cite[Proposition 1.1.11]{Ki_BT}). We also have a similar equivalence for the categories $\ModSGffr$ and $\ModSffr$ (\cite[Theorem 2.2.1]{CL}).


\subsection{Classification of finite flat group schemes and ramification}

Set $\tokbar=\okbar/p\okbar$ and $\bC$ to be the completion of $\Kbar$. Consider the ring
\[
R=\varprojlim (\tokbar \leftarrow \tokbar \leftarrow \cdots),
\]
where every transition map is the $p$-th power map. For an element $x=(x_0,x_1,\ldots)\in R$ with $x_n\in\tokbar$, we put $x^{(0)}=\lim_{n\to\infty}\hat{x}_{n}^{p^n}\in\cO_\bC$, where $\hat{x}_n$ is a lift of $x_n$ in $\okbar$. Then the ring $R$ is a complete valuation ring of characteristic $p$ with valuation $v_R(x)=v_p(x^{(0)})$ whose fraction field $\Frac(R)$ is algebraically closed, and the absolute Galois group $G_K=\Gal(\Kbar/K)$ naturally acts on this ring. We put $m_R^{\geq i}=\{x\in R\mid v_R(x)\geq i\}$. Define an element $\upi$ of $R$ with $v_R(\upi)=1/e$ by $\upi=(\pi,\pi_1,\pi_2,\ldots)$, where we abusively write $\pi_n$ also for its image in $\tokbar$. The ring $R$ has a natural $\SG$-algebra structure defined by the continuous map $\SG\to R$ which sends $u$ to the element $\upi$. Then we have the following classification theorem due to Breuil and Kisin.

\begin{thm}[\cite{Br}, \cite{Br_AZ}, \cite{Ki_Fcrys}, \cite{Ki_BT}]
There exists an anti-equivalence $\Gr(-)$ from the category $\ModSf$ ({\it resp.} $\ModSffr$) to the category of finite flat group schemes over $\okey$ killed by $p$ ({\it resp.} Barsotti-Tate groups over $\okey$). Moreover, put $\cG(-)=\Gr(\cM_\SG(-))$. Then for any object $\SGm$ of the category $\ModSGf$, we have a natural isomorphism of $G_{K_\infty}$-modules
\[
\varepsilon_\SGm: \cG(\SGm)(\okbar)|_{G_{K_\infty}}\to \TSG(\SGm)=\Hom_{\SG,\varphi}(\SGm, R).
\]
\end{thm}

The anti-equivalence $\cG(-)$ is compatible with Cartier duality in the following sense. For an object $\SGm$ of the category $\ModSGf$, we can define a natural dual object $\SGm^\vee$ (\cite{CL}, \cite{Li_FC}), as follows. Put $\SGm^\vee=\Hom_{\SG}(\SGm,\SG_1)$. Choose a basis $e_1,\ldots, e_h$ of the free $\SG_1$-module $\SGm$ and let $e^{\vee}_{1},\ldots,e^{\vee}_{h}$ denote its dual basis. Define a matrix $A\in M_h(\SG_1)$ by
\[
\varphi_{\SGm}(e_1,\ldots,e_h)=(e_1,\ldots,e_h)A.
\]
The $\varphi$-semilinear map $\varphi_{\SGm^\vee}:\SGm^\vee\to\SGm^\vee$ is defined to be
\[
\varphi_{\SGm^\vee}(e^{\vee}_{1},\ldots,e^{\vee}_{h})=(e^{\vee}_{1},\ldots,e^{\vee}_{h})(E(u)/c_0)({}^\mathrm{t}\!A)^{-1},
\]
which is independent of the choice of the basis $e_1,\ldots,e_h$. Then we have a natural isomorphism of finite flat group schemes over $\okey$
\[
\cG(\SGm)^\vee\to \cG(\SGm^\vee),
\]
where $\vee$ on the left-hand side means the Cartier dual (see \cite[Proposition 4.4]{Ha4}). This defines an isomorphism of $G_{K_\infty}$-modules
\[
\delta_\SGm: \cG(\SGm)^\vee(\okbar) \to \cG(\SGm^\vee)(\okbar)\overset{\varepsilon_{\SGm^\vee}}{\to}  \TSG(\SGm^\vee).
\]
Let $\SG_1^\vee$ be the object of the category $\ModSGf$ whose underlying $\SG$-module is $\SG_1$ and Frobenius map is given by $\varphi_{\SG_1^\vee}(1)=c_0^{-1}E(u)$. Then the natural pairing 
\[
\langle\ ,\ \rangle_\SGm: \SGm\times \SGm^\vee \to \SG_1^\vee
\]
induces a perfect pairing of $G_{K_\infty}$-modules
\[
\TSG(\SGm)\times \TSG(\SGm^\vee)\to \TSG(\SG_1^\vee),
\]
which is denoted also by $\langle\ ,\ \rangle_\SGm$. This pairing fits into a commutative diagram of $G_{K_\infty}$-modules
\[
\xymatrix{
\cG(\SGm)(\okbar)\times \cG(\SGm)^\vee(\okbar)\ar[r]\ar@<-6ex>[d]^{\wr}_{\varepsilon_\SGm}\ar@<+6ex>[d]^{\wr}_{\delta_\SGm} & \bZ/p\bZ(1) \ar[d]^{\wr}\\
\TSG(\SGm)\times\TSG(\SGm^\vee)\ar[r]_-{\langle\ ,\ \rangle_\SGm} & \TSG(\SG_1^\vee),
}
\]
where the top arrow is the Cartier pairing of $\cG(\SGm)$ (\cite[Proposition 4.4]{Ha4}).

On the other hand, for a $k[[u]]$-algebra $B$, we also have an anti-equivalence of categories $\cH(-)$ from $\mathrm{Mod}_{/B}^{1,\varphi}$ to the category of finite locally free group schemes $\cH$ over $B$ whose Verschiebung $V_\cH$ is zero such that the cokernel of the induced map $V_{\cH^\vee}:\Lie(\cH^\vee)^{(p)}\to \Lie(\cH^\vee)$ is killed by $u^e$ (\cite[Th\'eor\`{e}me 7.4]{SGA3-7A}). Note that the anti-equivalence $\cH(-)$ commutes with any base change. Suppose that the $B$-module $\SGm$ is free of rank $h$. Let $e_1,\ldots,e_h$ be its basis and put $\varphi(e_1,\ldots,e_h)=(e_1,\ldots,e_h)A$ for some matrix $A=(a_{i,j})\in M_h(B)$. Then the group scheme $\cH(\SGm)$ is by definition the additive group scheme whose affine algebra is
\[
B[X_1,\ldots,X_h]/(X_1^p-\sum_{j=1}^h a_{j,1}X_j,\ldots, X_h^p-\sum_{j=1}^h a_{j,h}X_j).
\]
Moreover, for the case of $B=k[[u]]$, we have a natural isomorphism of $G_{K_\infty}$-modules
\[
\cH(\SGm)(R)\to \Hom_{\SG,\varphi}(\SGm, R),
\]
by which we identify both sides. Hence we obtain the isomorphism of $G_{K_\infty}$-modules
\[
\cG(\SGm)(\okbar)|_{G_{K_\infty}}\to \cH(\SGm)(R).
\]
We normalize the indices of the ramification subgroup schemes of $\cH(\SGm)$ by multiplying the usual ones by $1/e$, so that
\[
\cH(\SGm)_i(R)=\Ker(\cH(\SGm)(R)\to\cH(\SGm)(R/m_R^{\geq i})). 
\]
Then we have the following correspondence of ramification between the finite flat group schemes $\cG(\SGm)$ and $\cH(\SGm)$.

\begin{thm}\label{ramcorr}
Let $\SGm$ be an object of the category $\ModSGf$.
\begin{enumerate}
\item\label{ramcorr1} (\cite[Theorem 1.1]{Ha4}) The natural isomorphism of $G_{K_\infty}$-modules
\[
\cG(\SGm)(\okbar)|_{G_{K_\infty}}\to \cH(\SGm)(R)
\]
preserves the upper and the lower ramification subgroups of both sides. 
\item\label{ramcorr2} (\cite[Corollary 4.6]{Ha4}) By the $k$-algebra isomorphism $k[[u^{1/p}]]/(u^e)\to \okp/p\okp$ sending $u^{1/p}$ to $\pi_1$, we identify both sides of the isomorphism. Then there exists an isomorphism of schemes
\[
\cG(\SGm)\times \Spec(\okp/p\okp) \to \cH(\SGm)\times \Spec(k[[u^{1/p}]]/(u^e))
\]
which preserves the zero section.
\end{enumerate}
\end{thm}

For a finite flat group scheme $\cG$ over $\okey$ which is killed by $p$, the upper and the lower ramification subgroups are in duality as in the following theorem.

\begin{thm}[\cite{Ti}, Theorem 1.6 or \cite{Fa}, Proposition 6]\label{TF}
Let $K/\bQ_p$ be an extension of complete discrete valuation fields and $\cG$ be a finite flat group scheme over $\okey$ killed by $p$. For $j\leq p/(p-1)$, we have the equality
\[
\cG^j(\okbar)^{\bot}=(\cG^\vee)_{l(j)+}(\okbar)
\]
of subgroups of $\cG^\vee(\okbar)$, where $\bot$ means the orthogonal subgroup with respect to the Cartier pairing and $l(j)=1/(p-1)-j/p$.
\end{thm}

We insert here a lemma which gives an upper bound of the lower ramification for finite flat group schemes killed by $p$. For an extension $\cK/k((u))$ of complete discrete valuation fields, let $v_u$ be the $u$-adic valuation on $\cK$ normalized as $v_u(u)=1$. Let $\cG$ be a finite flat generically etale group scheme over $\mathcal{O}_{\cK}$. Fix a positive integer $m$ and we normalize the indices of the ramification subgroup schemes of $\cG$ by multiplying the usual ones by $1/m$. Write $\omega_\cG\simeq \oplus_i \mathcal{O}_{\cK}/(a_i)$ for some $a_i \in \mathcal{O}_{\cK}$ and put $\deg(\cG)=m^{-1}\sum_iv_u(a_i)$. In particular, for an object $\SGm$ of the category $\ModSGf$, we define $\deg(\cH(\SGm))$ by putting $m=e$.

\begin{lem}\label{lowramdeg}
Let $\cK/\bQ_p$ ({\it resp.} $\cK/k((u))$) be an extension of complete discrete valuation fields and $\cG$ be a finite flat generically etale group scheme over $\cO_\cK$ killed by $p$. Then $\cG_i=0$ for any $i>\deg(\cG)/(p-1)$.
\end{lem}
\begin{proof}
Let $\cK^{\mathrm{sep}}$ be a separable closure of $\cK$. We may assume $\cG(\cO_{\cK^{\mathrm{sep}}})=\cG(\cO_\cK)$. For $i$ as in the lemma, put $x\in \cG_i(\cO_\cK)$ and let $\cH$ be the scheme-theoretic closure of the subgroup $\bF_px\subseteq \cG(\cO_\cK)$ in $\cG$. Then we have $\deg(\cH)\leq \deg(\cG)$. By the Oort-Tate classification (\cite{OT}), the affine algebra of $\cH$ is isomorphic to the ring $\cO_\cK[X]/(X^p-aX)$ for some $a\in \cO_\cK$ with $\deg(\cH)=v_p(a)$ ({\it resp.} $\deg(\cH)=m^{-1}v_u(a)$). Hence we obtain $\cH_i=0$ and $x=0$.
\end{proof}


\subsection{Hodge filtration and Breuil-Kisin modules}\label{Hfil}

Finally, due to the lack of references, we explain how to decode the Hodge filtration, the Hodge height and the Hodge-Tate map for a truncated Barsotti-Tate group of level one over $\okey$ from its corresponding Breuil-Kisin module. Put $\sS_1=\Spec(\tokey)$ and $E_1=\Spec(S_1)$. Consider the big crystalline site with the fppf topology $(\sS_1/E_1)_\CRYS$, as in \cite{BBM}. For an fppf sheaf $\cE$ over $\Spec(\okey)$, we let $\cE_1$ denote its restriction to $\Spec(\tokey)$ only in this subsection, and for an fppf sheaf $\cF$ over $\Spec(\tokey)$, let $\underline{\cF}$ denote the associated sheaf on the site $(\sS_1/E_1)_\CRYS$ (\cite[(1.1.4.5)]{BBM}). We also write the crystalline Dieudonn\'{e} functor as  $\bD^{*}(-)=\sExt^1_{\sS_1/E_1}(\underline{-},\cO_{\sS_1/E_1})$ (\cite[D\'{e}finition 3.1.5]{BBM}). Let $v_u$ be the $u$-adic valuation on the ring $k[[u]]$ with $v_u(u)=1$, as before. By the $k$-algebra isomorphism $k[[u]]/(u^e)\to \tokey$ sending $u$ to $\pi$, we identify both sides of the isomorphism.

\begin{lem}\label{invdiff}
\begin{enumerate}
\item \label{invdiff-BT}
Let $\cG$ be a truncated Barsotti-Tate group of level one over $\okey$ and $\cM$ be the object of $\ModSf$ which corresponds to $\cG$ via the anti-equivalence $\Gr(-)$. Then there exist natural isomorphisms of $\tokey$-modules
\[
\Fil^1\cM/(\Fil^1S)\cM\to \omega_{\cG_1},\quad \cM/\Fil^1\cM\to \Lie(\cG^\vee_1).
\]
\item \label{invdiff-p}
Let $\cG$ be a finite flat group scheme over $\okey$ killed by $p$ and $\cM$ be its corresponding object of the category $\ModSf$. Then we have a natural isomorphism of $\tokey$-modules $\Fil^1\cM/(\Fil^1S)\cM\to \omega_{\cG_1}$. 

\item \label{invdiff-deg}
Let $\SGm$ be an object of the category $\ModSGf$. Then we have $\deg(\cG(\SGm))=\deg(\cH(\SGm))=e^{-1}v_u(\det(\varphi_\SGm))$.
\end{enumerate}
\end{lem}
\begin{proof}
Let us consider the assertion (\ref{invdiff-BT}). By \cite[Th\'{e}or\`{e}me 4.4 (e)]{Il}, we can find a Barsotti-Tate group $\Gamma$ over $\okey$ such that its $p$-torsion subgroup scheme $\Gamma[p]$ is isomorphic to $\cG$. Let $\cN$ be the object of the category $\ModSffr$ corresponding to $\Gamma$ and put $\cN_1=\cN/p\cN$, which is naturally considered as an object of the category $\ModSf$. By the construction of the anti-equivalence (\cite{Ki_Fcrys}), we have a natural isomorphism of exact sequences
\[
\xymatrix{
0 \ar[r] & \Fil^1\cN_1/(\Fil^1S)\cN_1 \ar[r]\ar[d]_{\wr} & \cN_1/(\Fil^1S)\cN_1 \ar[r]\ar[d]_{\wr}  & \cN_1/\Fil^1\cN_1 \ar[r]\ar[d]_{\wr} & 0\\
0 \ar[r] & \omega_{\Gamma_1} \ar[r] & \bD^{*}(\Gamma_1)_{\tokey} \ar[r] & \Lie(\Gamma^\vee_1) \ar[r] & 0,
}
\]
where $\bD^{*}(\Gamma_1)_{\tokey}$ is the set of sections on the divided power thickening $\tokey \to \tokey$. Thus we also have the natural isomorphism
\[
\xymatrix{
0 \ar[r] & \Fil^1\cM/(\Fil^1S)\cM \ar[r]\ar[d]_{\wr} & \cM/(\Fil^1S)\cM \ar[r]\ar[d]_{\wr}  & \cM/\Fil^1\cM \ar[r]\ar[d]_{\wr} & 0\\
0 \ar[r] & \omega_{\cG_1} \ar[r] & \bD^{*}(\cG_1)_{\tokey} \ar[r] & \Lie(\cG_1^\vee) \ar[r] & 0.
}
\]
 
For the assertion (\ref{invdiff-p}), take a resolution 
\[
0\to \cG\to\Gamma\to \Gamma'\to 0
\]
of $\cG$ by Barsotti-Tate groups $\Gamma$ and $\Gamma'$ over $\okey$. Let $\cN$ and $\cN'$ be the objects of the category $\ModSffr$ corresponding to $\Gamma$ and $\Gamma'$, respectively. Then we have exact sequences
\[
0\to \omega_{\Gamma'}\to \omega_{\Gamma}\to \omega_{\cG}\to 0,\quad 0\to \cN'\to \cN \to \cM \to 0.
\]
By tensoring $\tokey$ to the first sequence and using the assertion (\ref{invdiff-BT}), we obtain an exact sequence
\[
\Fil^1\cN'/(p\Fil^1\cN'+(\Fil^1S)\cN')\to \Fil^1\cN/(p\Fil^1\cN+(\Fil^1S)\cN)\to \omega_{\cG_1}\to 0.
\]
This and the second sequence induce the isomorphism $\Fil^1\cM/(\Fil^1S)\cM\to \omega_{\cG_1}$.

For the assertion (\ref{invdiff-deg}), choose a basis $e_1,\ldots,e_h$ of $\SGm$ and put
\[
\varphi(e_1,\ldots,e_h)=(e_1,\ldots,e_h)A
\]
for some $A\in M_h(k[[u]])$. Let $u^{s_1},\ldots,u^{s_h}$ be the elementary divisors of $A$, which satisfy $0\leq s_i\leq e$ for any $i$. Note that the matrix $u^eA^{-1}$ is contained in $M_h(k[[u]])$ and its elementary divisors are $u^{e-s_1},\ldots,u^{e-s_h}$. By the definition of the functor $\cM_\SG(-)$, we have the equality
\[
\Fil^1\cM_\SG(\SGm)/(\Fil^1S)\cM_\SG(\SGm)=\Span_{\SG_1}((1\otimes e_1,\ldots,1\otimes e_h)u^eA^{-1}).
\]
Thus, by the assertion (\ref{invdiff-p}), we obtain an isomorphism 
\[
\omega_{\cG(\SGm)_1}\simeq \oplus_i u^{e-s_i}k[[u]]/u^ek[[u]],
\]
which implies the equalities $\deg(\cG(\SGm))=\sum_i e^{-1}s_i=e^{-1}v_u(\det\varphi)$. From the explicit description of the affine algebra of $\cH(\SGm)$ given before, we see that this is also equal to $\deg(\cH(\SGm))$.
\end{proof}

Let $\cG$ and $\cM$ be as in Lemma \ref{invdiff} (\ref{invdiff-BT}) and set $\SGm$ to be the object of $\ModSGf$ which corresponds to $\cM$ via the equivalence $\cM_\SG(-)$. We let $h$ and $d$ denote the height and the dimension of $\cG$, respectively. Let $e_1,\ldots,e_h$ be a basis of $\SGm$ and $A$ be the element of $M_h(k[[u]])$ with
\[
\varphi(e_1,\ldots,e_h)=(e_1,\ldots,e_h)A.
\]
We also put $\SGm_1=\SGm/u^e\SGm$ and $\Fil^1\SGm_1=(1\otimes\varphi)(\tokey\otimes_{\varphi,\tokey}\SGm_1)\subseteq \SGm_1$.
Then we have an isomorphism of $\tokey$-modules
\[
\cM/(\Fil^1S)\cM\to \tokey\otimes_{\varphi,\tokey}\SGm_1.
\]
From Lemma \ref{invdiff} (\ref{invdiff-BT}) and the definition of the functor $\cM_\SG(-)$, we also have natural isomorphisms of $\tokey$-modules
\[
\Lie(\cG^\vee_1)\to \cM/\Fil^1\cM \overset{1\otimes\varphi}\to \Fil^1\SGm_1.
\]
Hence the $\tokey$-module $\Fil^1\SGm_1$ is free of rank $h-d$ and each elementary divisor of the matrix $A$ is either $1$ or $u^e$. This implies that the $\tokey$-module $\SGm_1/\Fil^1\SGm_1$ is free of rank $d$.

Consider the diagram of $\tokey$-modules
\[
\xymatrix{
\tokey \otimes_{\varphi,\tokey}(\tokey \otimes_{\varphi,\tokey} \SGm_1) \ar[r]^-{1\otimes1\otimes \varphi}\ar[d]_{1\otimes \varphi \otimes \varphi} & \tokey \otimes_{\varphi,\tokey}\Fil^1\SGm_1 \ar[r]\ar[d] & \tokey \otimes_{\varphi,\tokey}\SGm_1 \ar[d]^{1\otimes \varphi}\\
\tokey \otimes_{\varphi,\tokey} \SGm_1 \ar[r]_-{1\otimes \varphi} & \Fil^1\SGm_1 \ar[r] & \SGm_1
}
\] 
whose left horizontal arrows are surjections and right horizontal arrows are natural inclusions. The left vertical arrow is the map induced by the Frobenius $F_{\cG_1}$ via the natural isomorphism $\cM\simeq \bD^{*}(\cG_1)(S_1\to\tokey)$ and the middle vertical arrow is the map induced by the left vertical arrow. Lemma \ref{invdiff} (\ref{invdiff-BT}) implies that the truncated valuation of the determinant of the latter map is equal to $\Hdg(\cG)$, since the natural isomorphism $\Hom(\cG_1,\Ga)\to \Lie(\cG^\vee_1)$ takes the action of the Frobenius $F_{\cG_1}$ on the left-hand side to the action of the Verschiebung $V_{\cG^\vee_1}$ on the right-hand side. We see that the outer square is commutative and thus the right square is also commutative. Hence we have an exact sequence
\[
0 \to \Fil^1\SGm_1 \to \SGm_1 \to \SGm_1/\Fil^1\SGm_1 \to 0
\]
of the category $\mathrm{Mod}_{/\tokey}^{1,\varphi}$ with $v_p(\det(\varphi_{\Fil^1\SGm_1}))=\Hdg(\cG)$.

Let $i\leq 1$ be a non-negative rational number. Then the zeroth projection $\prjt_0:R\to \tokbar$ induces an isomorphism $R/m_R^{\geq i}\to \okbar/m_{\Kbar}^{\geq i}$, by which we identify both sides. Moreover, we also have natural isomorphisms
\[
\Hom_{\tokey}(\Fil^1\SGm_1, R/m_R^{\geq i})\to \Hom_{\tokey}(\Lie(\cG^\vee_1), \okbar/m_{\Kbar}^{\geq i})\to \omega_{\cG_1^\vee}\otimes \okbar/m_{\Kbar}^{\geq i}.
\]
\begin{lem}\label{HTmap}
Let $\cG$ be a truncated Barsotti-Tate group of level one over $\okey$ and $\SGm$ be the object of $\ModSGf$ which corresponds to $\cG$ via the anti-equivalence $\cG(-)$. 
Then the composite
\[
\cG(\okbar)\to \cH(\SGm)(R)\to \Hom_{\tokey}(\Fil^1\SGm_1, R/m_R^{\geq i})\to \omega_{\cG_1^\vee}\otimes \okbar/m_{\Kbar}^{\geq i}
\]
coincides with the Hodge-Tate map $\HT_i$ for any $i\leq 1$.
\end{lem}
\begin{proof}
For a sheaf $\cE$ on the site $(\sS_1/E_1)_\CRYS$, we let $\cE_{\tokbar}$ denote the set of sections on the natural divided power thickening $\tokbar\to \tokbar$.
For a valued point $x:\Spec(\okbar)\to \cG$, we let $x^\vee:\cG^\vee\times\Spec(\okbar)\to \mu_p$ denote its dual map. Since $\cG$ is a truncated Barsotti-Tate group of level one, we have a commutative diagram
\[
\xymatrix{
\Hom(\bD^{*}(\bZ/p\bZ)_{\tokbar}, \bD^{*}(\bZ/p\bZ)_{\tokbar}) \ar[d]_{\wr} \ar[r]^-{\bD^{*}(x)^\vee} & \Hom(\bD^{*}(\cG_1)_{\tokbar}, \bD^{*}(\bZ/p\bZ)_{\tokbar}) \ar[d]_{\wr} \\
\bD^{*}(\mu_p)_{\tokbar} \ar[r]_{\bD^{*}(x^{\vee})}& \bD^{*}(\cG^\vee_1)_{\tokbar}
}
\]
where the vertical arrows are isomorphisms compatible with Hodge filtrations (\cite[Proposition 5.3.6]{BBM}). Put $\cM=\cM_{\SG}(\SGm)$. In particular, by Lemma \ref{invdiff} (\ref{invdiff-BT}) we also have a commutative diagram whose vertical arrows are isomorphisms
\[
\xymatrix{
\Hom_S(\cM/\Fil^1\cM,\tokbar) \ar@{^{(}->}[r]\ar[d]^{\wr} & \Hom_S(\cM,\tokbar)\ar[d]^{\wr}\\
\omega_{\cG^\vee_1}\otimes\tokbar \ar@{^{(}->}[r]& \bD^{*}(\cG^\vee_1)_{\tokbar}.
}
\]
Then $\HT_1(x)$ coincides with the image of the identity map on the upper left corner of the former diagram by the lower composite, which is contained in the submodule $\omega_{\cG^\vee_1}\otimes\tokbar$.

Now we consider a diagram
\[
\xymatrix{
\cG(\okbar)\ar[r]\ar[rdd] \ar@/^1pc/@{-->}[rr]&   \Hom_S(\cM,R^\PD)\ar[d]& \Hom_{\SG,\varphi}(\SGm,R) \ar[l]\ar[d]\\
 & \Hom_S(\cM,\tokbar)& \Hom_{\SG,\varphi}(\SGm, \tokbar)\ar[l]\ar[d]\\
 & \Hom_S(\cM/\Fil^1\cM,\tokbar)\ar[u] & \Hom_{\tokbar}(\Fil^1\SGm_1,\tokbar).\ar[l]_-{\sim}
}
\]
Here $R^\PD$ is the divided power envelope of the ring $R$ with respect to the ideal $m_R^{\geq 1}$. The right-pointing solid arrows are defined by $x\mapsto \bD^{*}(x)$ on the sets of sections over the divided power thickenings $R^\PD\to \tokbar$ and $\tokbar\to \tokbar$. The top left-pointed horizontal arrow maps a homomorphism $f:\SGm\to R$ to the homomorphism $\cM=S\otimes_{\varphi,\SG}\SGm \to R^\PD$ defined by $s\otimes m\mapsto s\varphi(f(m))$ and the middle left-pointed horizontal arrow is defined similarly. The bottom one is induced by the isomorphism $1\otimes \varphi: \cM/\Fil^1\cM\to \Fil^1\SGm_1$. The dotted arrow is the isomorphism $\varepsilon_\SGm:\cG(\okbar)\to \TSG(\SGm)$. From the definition, we see that the bottom square is commutative and thus the diagram is also commutative. Hence the lemma follows.
\end{proof}



\section{Level one canonical subgroup}

In this section, we prove the following level one case of Theorem \ref{main}. 

\begin{thm}\label{cansubmain}
Let the notation be as in Section \ref{intro} and suppose $p>2$. Let $\cG$ be a truncated Barsotti-Tate group of level one, height $h$ and dimension $d$ over $\okey$ with $0<d<h$ and Hodge height $w=\Hdg(\cG)$. 
\begin{enumerate}
\item\label{cansubmain-1} If $w<p/(p+1)$, then there exists a unique finite flat closed subgroup scheme $\cC$ of $\cG$ of order $p^d$ over $\okey$ such that $\cC\times \Spec(\okey/m_K^{\geq 1-w})$ coincides with the kernel of the Frobenius homomorphism of $\cG\times \Spec(\okey/m_{K}^{\geq 1-w})$. We refer the subgroup scheme $\cC$ as the canonical subgroup of $\cG$. Moreover, the subgroup scheme $\cC$ has the following properties:
\begin{enumerate}
\item\label{cansubmain-deg} $\deg(\cG/\cC)=w$.
\item\label{cansubmain-iso} Let $\cC'$ be the canonical subgroup of $\cG^\vee$. Then we have the equality of subgroup schemes $\cC'=(\cG/\cC)^\vee$, or equivalently $\cC(\okbar)=\cC'(\okbar)^\bot$, where $\bot$ means the orthogonal subgroup with respect to the Cartier pairing.
\item\label{cansubmain-ram1} $\cC=\cG_{(1-w)/(p-1)}=\cG^{pw/(p-1)+}$.
\end{enumerate}
\item\label{cansubmain-HT} If $w<(p-1)/p$, then the subgroup $\cC(\okbar)$ coincides with the kernel of the Hodge-Tate map $\HT_b: \cG(\okbar)\to \omega_{\cG^\vee}\otimes \okbar/m_{\Kbar}^{\geq b}$ for $w/(p-1) <b \leq 1-w$.
\item\label{cansubmain-ram2} If $w<1/2$, then $\cC$ coincides both with the lower ramification subgroup scheme $\cG_b$ for $w/(p-1)<b\leq (1-w)/(p-1)$ and the upper ramification subgroup scheme $\cG^{j+}$ for $p w/(p-1)\leq j<p(1-w)/(p-1)$. 
\end{enumerate}
\end{thm}

\begin{proof}
Let $W(\bar{k})$ be the Witt ring of an algebraic closure $\bar{k}$ of $k$. Choose a Cohen ring $C(k)$ with residue field $k$ and inclusions $C(k)\to W(\bar{k})$ and $C(k)\to \okey$. Then the $\okey$-algebra $\mathcal{O}_{K'}=\okey\otimes_{C(k)}W(\bar{k})$ is a complete discrete valuation ring with residue field $\bar{k}$ and relative ramification index one over $\okey$. Note that finite flat closed subgroup schemes $\cD$ and $\cD'$ of $\cG$ of the same order over $\okey$ coincide if and only if the map $\cD\to \cG/\cD'$ is zero and the latter can be checked after the base change to $\mathcal{O}_{K'}$. Hence if we show the theorem for $\cG\times\Spec(\mathcal{O}_{K'})$, then the upper ramification subgroup scheme $\cC=\cG^{pw/(p-1)+}$ satisfies all the properties in the theorem. Therefore, replacing $\okey$ with $\mathcal{O}_{K'}$, we may assume that $k$ is perfect. 

We adopt the notation of the previous section. Let $\SGm$ be the object of the category $\ModSGf$ corresponding to $\cG$ via the anti-equivalence $\cG(-)$. Put $\SGm_1=\SGm/u^e\SGm$ and $\Fil^1\SGm_1=(1\otimes\varphi)(\tokey\otimes_{\varphi,\tokey}\SGm_1)$. We identify the $k$-algebras $k[[u]]/(u^e)$ and $\tokey$ by $u\mapsto \pi$, as before. As we have seen in Subsection \ref{Hfil}, we have the exact sequence
\[
0\to \Fil^1\SGm_1 \to \SGm_1 \to \SGm_1/\Fil^1\SGm_1 \to 0
\]
of the category $\mathrm{Mod}_{/\tokey}^{1,\varphi}$, where the $\tokey$-module $\Fil^1\SGm_1$ ({\it resp.} $\SGm_1/\Fil^1\SGm_1$) is free of rank $h-d$ ({\it resp.} rank $d$) and the equality $v_p(\det(\varphi_{\Fil^1\SGm_1}))=w$ holds.

We choose once and for all a basis $e_1\ldots,e_h$ of $\SGm$ such that $e_1,\ldots, e_{h-d}$ is a lift of a basis of the $\tokey$-module $\Fil^1\SGm_1$ to the $k[[u]]$-submodule $(1\otimes \varphi)(\SG_1\otimes_{\varphi, \SG_1}\SGm)$ and $e_{h-d+1},\ldots,e_h$ is a lift of a basis of the $\tokey$-module $\SGm_1/\Fil^1\SGm_1$. Then the elements $e_1,\ldots,e_{h-d},u^e e_{h-d+1},\ldots,u^e e_h$ form a basis of the $k[[u]]$-module $(1\otimes \varphi)(\SG_1\otimes_{\varphi, \SG_1}\SGm)$ and we have
\[
\varphi(e_1,\ldots,e_h)=(e_1,\ldots,e_h)\begin{pmatrix}P_1&P_2\\ u^eP_3&u^eP_4 \end{pmatrix}
\]
for some matrices $P_i$ with entries in the ring $k[[u]]$, where $P_4$ is a $d\times d$-matrix. Since the elements on the right-hand side also form a basis of $(1\otimes \varphi)(\SG_1\otimes_{\varphi, \SG_1}\SGm)$, the matrix
\[
\begin{pmatrix}
P_1 & P_2\\
P_3 & P_4
\end{pmatrix}
\]
is contained in $GL_h(k[[u]])$. Moreover, since $w<1$, we also have the equality $v_u(\det(P_1))=ew$ and thus there exists a matrix $\hat{P}_1\in M_{h-d}(k[[u]])$ satisfying $P_1\hat{P}_1=u^{ew} I_{h-d}$, where $I_{h-d}$ is the identity matrix. Here we note that the number $e w$ is a non-negative integer. 

For the uniqueness assertion in (\ref{cansubmain-1}), we first show the following lemma.

\begin{lem}\label{Froblem}
Let $\cD$ be a finite flat closed subgroup scheme of $\cG$ over $\okey$ and $\SGl$ be the subobject of $\SGm$ in the category $\ModSGf$ corresponding to the quotient $\cG/\cD$ via the anti-equivalence $\cG(-)$. Then, for any $i\in \frac{1}{e}\bZ$ with $0< i \leq 1$, the following conditions are equivalent:
\begin{enumerate}
\item\label{uni-1} The subgroup scheme $\cD\times \Spec(\okey/m_K^{\geq i})$ coincides with the kernel of the Frobenius homomorphism of $\cG \times \Spec(\okey/m_K^{\geq i})$.
\item\label{uni-2} $\SGl/u^{ei}\SGl=\Fil^1\SGm_1/u^{ei}\Fil^1\SGm_1$. 
\end{enumerate}
\end{lem}
\begin{proof}
Put $\SGn=\SGm/\SGl$, which is the object of the category $\ModSGf$ corresponding to the finite flat group scheme $\cD$. Note that the Frobenius map is compatible with any morphism between schemes of characteristic $p$, and that the kernel of the Frobenius of the truncated Barsotti-Tate group $\cG \times \Spec(\okey/m_K^{\geq i})$ is a finite flat closed subgroup scheme of order $p^d$. By Theorem \ref{ramcorr} (\ref{ramcorr2}), we see that the kernel of the Frobenius of the group scheme $\cH(\SGm) \times \Spec(k[[u]]/(u^{ei}))$ is also finite flat of rank $p^d$. Now the condition (\ref{uni-1}) in the lemma is equivalent to saying that $\cD\times \Spec(\okey/m_K^{\geq i})$ is killed by the Frobenius and of order $p^d$. By Theorem \ref{ramcorr} (\ref{ramcorr2}), this holds if and only if $\cH(\SGn) \times \Spec(k[[u]]/(u^{ei}))$ is killed by the Frobenius and of order $p^d$, namely if the latter subgroup scheme coincides with the kernel of the Frobenius of $\cH(\SGm) \times \Spec(k[[u]]/(u^{ei}))$.

From the definition of the anti-equivalence $\cH(-)$, we see that the Frobenius of $\cH(\SGm/u^{ei}\SGm)$ corresponds to the natural map 
\[
1\otimes \varphi: k[[u]]/(u^{e i})\otimes_{\varphi,k[[u]]}\SGm/u^{e i}\SGm \to \SGm/u^{e i}\SGm.
\]
Since $\Coker(1\otimes \varphi)=\SGm_1/\Fil^1\SGm_1\otimes k[[u]]/(u^{e i})$ is free of finite rank over $k[[u]]/(u^{e i})$, the kernel of the Frobenius coincides with $\cH(\Coker(1\otimes \varphi))$. Thus the subgroup scheme $\cH(\SGn/u^{e i}\SGn)$ coincides with the kernel of the Frobenius of $\cH(\SGm/u^{e i}\SGm)$ if and only if $\SGl/u^{e i}\SGl=\Img(1\otimes \varphi)=\Fil^1\SGm_1/u^{e i}\Fil^1\SGm_1$.
\end{proof}

By this lemma, the unique existence of $\cC$ as in Theorem \ref{cansubmain} (\ref{cansubmain-1}) follows from the lemma below, by putting $\SGn=\SGm/\SGl$ and $\cC=\cG(\SGn)$.

\begin{lem}\label{uniquelift}
There exists a unique direct summand $\SGl$ of the $k[[u]]$-module $\SGm$ which is free of rank $h-d$ such that $\SGl$ is stable under $\varphi=\varphi_\SGm$ and $\SGl$ satisfies the condition (\ref{uni-2}) of Lemma \ref{Froblem} for $i=1-w$. Moreover, the $\varphi$-module $\SGl$ defines a subobject of $\SGm$ in the category $\ModSGf$ with $v_u(\det(\varphi_{\SGl}))=ew$.
\end{lem}
\begin{proof}
Let $\SGl$ be a direct summand of $\SGm$ satisfying the condition (\ref{uni-2}) of Lemma \ref{Froblem} for $i=1-w$. Let $\delta_1,\ldots, \delta_{h-d}$ be a basis of $\SGl$. Then we have 
\[
(\delta_1,\ldots,\delta_{h-d})=(e_1,\ldots,e_h)\begin{pmatrix}I_{h-d}+u^{e(1-w)}B'\\ u^{e(1-w)}B \end{pmatrix}
\]
with some $B\in M_{d, h-d}(k[[u]])$ and $B'\in M_{h-d}(k[[u]])$. By multiplying the inverse of the invertible matrix $I_{h-d}+u^{e(1-w)}B'$, we may assume $B'=0$. It is enough to show that there exists $B$ uniquely such that the resulting $\SGl$ is stable under $\varphi=\varphi_{\SGm}$, and that $\SGl$ defines an element of the category $\ModSGf$ satisfying $v_u(\det\varphi_\SGl)=ew$.

Note that we have
\[
\varphi(\delta_1,\ldots,\delta_{h-d})=(e_1,\ldots,e_h)\begin{pmatrix}P_1&P_2\\ u^eP_3&u^eP_4\end{pmatrix}\begin{pmatrix}I_{h-d}\\ u^{ep(1-w)}\varphi(B) \end{pmatrix}.
\]
Consider the equation
\[
\varphi(\delta_1,\ldots,\delta_{h-d})=(\delta_1,\ldots,\delta_{h-d})D
\]
for $D\in M_{h-d}(k[[u]])$. This is equivalent to the following equations:
\[
\left\{
\begin{array}{l}
P_1+u^{ep(1-w)}P_2\varphi(B)=D,\\
u^eP_3+u^{e+ep(1-w)}P_4\varphi(B)=u^{e(1-w)}BD.
\end{array}
\right.
\]
From this we obtain the equation for $B$
\[
BP_1=u^{ew}P_3-u^{ep(1-w)}BP_2\varphi(B)+u^{ew+ep(1-w)}P_4\varphi(B).
\]
By multiplying $\hat{P}_1$, we have
\[
B=P_3\hat{P}_1-u^{ep(1-w)-ew}BP_2\varphi(B)\hat{P}_1+u^{ep(1-w)}P_4\varphi(B)\hat{P}_1.
\]
The assumption $w<p/(p+1)$ implies the inequalities $ep(1-w)-ew>0$ and $ep(1-w)>0$. Therefore we can solve this equation uniquely by recursion to obtain $B$ and $D$ satisfying the above equations. Moreover, we also have $D=P_1(I_{h-d}+u^{ep(1-w)-ew}\hat{P}_1P_2\varphi(B))$ and the matrix $I_{h-d}+u^{ep(1-w)-ew}\hat{P}_1P_2\varphi(B)$ is invertible. Hence we see that the module $\SGl$ defines an object of the category $\ModSGf$ and $v_u(\det(\varphi_\SGl))=ew$. This concludes the proof of the lemma.
\end{proof}

By Lemma \ref{invdiff} (\ref{invdiff-deg}), we obtain the equalities $\deg(\cG/\cG(\SGn))=\deg(\cG(\SGl))=e^{-1}v_u(\det(\varphi_\SGl))=w$ and the part (\ref{cansubmain-deg}) of Theorem \ref{cansubmain} follows.

For the part (\ref{cansubmain-iso}), put $\bar{\cG}=\cG\times\Spec(\okey/m_K^{\geq 1-w})$. Since $\bar{\cG}$ is a truncated Barsotti-Tate group of level one, we have the equalities of finite flat closed subgroup schemes 
\[
(\Img(F_{\bar{\cG}}))^\vee=\Img(V_{\bar{\cG}^\vee})=\Ker(F_{\bar{\cG}^\vee}).
\]
Thus the subgroup scheme $(\cG/\cC)^\vee \times \Spec(\okey/m_K^{\geq 1-w})$ coincides with the kernel of the Frobenius of $\bar{\cG}^\vee$. By the uniqueness of the canonical subgroup we have just proved, we obtain the equality $\cC'=(\cG/\cC)^\vee$.

Let us prove the part (\ref{cansubmain-ram1}). By Theorem \ref{TF} and the part (\ref{cansubmain-iso}), it suffices to show the equality $\cC=\cG_{(1-w)/(p-1)}$. By Theorem \ref{ramcorr} (\ref{ramcorr1}), we are reduced to show the equality $\cH(\SGn)=\cH(\SGm)_{(1-w)/(p-1)}$. For this, we identify an element of the group $\cH(\SGm)(R)=\Hom_{\SG_1,\varphi}(\SGm,R)$ defined by $e_i\mapsto x_i$ for $1\leq i\leq h-d$ and $e_{h-d+i}\mapsto y_i$ for $1\leq i \leq d$ with a solution $(\underline{x},\underline{y})=(x_1,\ldots,x_{h-d},y_1,\ldots,y_d)$ in $R^h$ of the equation
\[
(\underline{x}^p,\underline{y}^p)=(\underline{x},\underline{y})\begin{pmatrix}P_1&P_2\\ u^eP_3&u^eP_4\end{pmatrix},
\]
where we put $\underline{x}^p=(x_1^p,\ldots,x_{h-d}^p)$ and similarly for $\underline{y}^p$. We use the notation in the proof of Lemma \ref{uniquelift}. Then we can also identify the group $\cH(\SGl)(R)$ with the set of solutions $\underline{z}=(z_1,\ldots,z_{h-d})$ in $R^{h-d}$ of the equation
\[
\underline{z}^p=\underline{z}D
\]
and the natural map $\cH(\SGm)\to \cH(\SGl)$ is identified with the map $(\underline{x},\underline{y})\mapsto \underline{z}=\underline{x}+u^{e(1-w)}\underline{y}B$.

Let $(\underline{x}, \underline{y})$ be an element of $\cH(\SGn)(R)=\Ker(\cH(\SGm)(R)\to \cH(\SGl)(R))$. Then we have $\underline{x}+u^{e(1-w)}\underline{y}B=0$ and thus, by eliminating $\underline{x}$ from the equation for $(\underline{x},\underline{y})$, we obtain
\[
\underline{y}^p=u^{e(1-w)}\underline{y}(u^{ew}P_4-BP_2).
\]
Hence the inequality 
\[
p\min_{1\leq i \leq d}v_R(y_i) \geq 1-w + \min_{1\leq i \leq d}v_R(y_i)
\]
follows, which implies $v_R(y_i)\geq (1-w)/(p-1)$ and thus $v_R(x_i)\geq p(1-w)/(p-1)$ for any $i$. This shows the inclusion $\cH(\SGn)(R) \subseteq \cH(\SGm)_{(1-w)/(p-1)}(R)$.

Conversely, let $(\underline{x}, \underline{y})$ be an element of $\cH(\SGm)_{(1-w)/(p-1)}(R)$. Then we can write $(\underline{x}, \underline{y})=u^{e(1-w)/(p-1)}(\underline{a}, \underline{b})$ with some tuples $\underline{a}, \underline{b}$ in $R$ and a $(p-1)$-st root $u^{e(1-w)/(p-1)}$ of $u^{e(1-w)}$ in $R$. These tuples satisfy the equality
\[
u^{e(1-w)}(\underline{a}^p, \underline{b}^p)=(\underline{a}, \underline{b})\begin{pmatrix}P_1&P_2\\ u^eP_3&u^eP_4\end{pmatrix}.
\]
Let $\begin{pmatrix}Q_1&Q_2\\ Q_3&Q_4\end{pmatrix}$ be the inverse of the matrix $\begin{pmatrix}P_1&P_2\\ P_3&P_4\end{pmatrix}$. By multiplying this, we obtain
\[
\underline{a}=u^{e(1-w)}(\underline{a}^pQ_1+\underline{b}^pQ_3)
\]
and thus $v_R(x_i)\geq p(1-w)/(p-1)$ for any $i$. Let $\underline{z}$ be the image of the element $(\underline{x}, \underline{y})$ by the natural map $\cH(\SGm)(R)\to \cH(\SGl)(R)$. Then we have $v_R(z_i)\geq p(1-w)/(p-1) >w/(p-1)$. By Lemma \ref{invdiff} (\ref{invdiff-deg}), the equality $\deg(\cH(\SGl))=w$ holds and Lemma \ref{lowramdeg} implies $\underline{z}=0$. Hence the inclusion $\cH(\SGm)_{(1-w)/(p-1)}(R) \subseteq \cH(\SGn)(R)$ follows and we conclude the proof of the part (\ref{cansubmain-ram1}).

For the assertion (\ref{cansubmain-HT}) of Theorem \ref{cansubmain}, let us assume $w<(p-1)/p$. Let $b$ be a rational number with $w/(p-1)< b \leq 1-w$. By the definition of the submodule $\SGl$, we have a commutative diagram
\[
\xymatrix{
\cH(\SGm)(R) \ar[r]\ar[d]& \cH(\SGl)(R)\ar[d]  \\
\cH(\SGm/u^{e(1-w)}\SGm)(R/m_R^{\geq b}) \ar[r]\ar[rd]& \cH(\SGl/u^{e(1-w)}\SGl)(R/m_R^{\geq b})\ar@{=}[d]\\
& \cH(\Fil^1\SGm_1/u^{e(1-w)}\Fil^1\SGm_1)(R/m_R^{\geq b})
}
\]
whose upper right vertical arrow is an injection by Lemma \ref{lowramdeg}. Therefore the assertion (\ref{cansubmain-HT}) follows from Lemma \ref{HTmap}.

Finally, let us prove the assertion (\ref{cansubmain-ram2}). Assume $w<1/2$. Take $i>w/(p-1)$ and consider the commutative diagram with exact upper row
\[
\xymatrix{
0 \ar[r] & \cC(\okbar) \ar[r] & \cG(\okbar)\ar[r]\ar[d] & (\cG/\cC)(\okbar) \ar[r]\ar[d] & 0 \\
         &            &    \cG(\okbar/m_{\Kbar}^{\geq i}) \ar[r] & (\cG/\cC)(\okbar/m_{\Kbar}^{\geq i}).
}
\]
By Lemma \ref{lowramdeg}, we have $(\cG/\cC)_i=0$ and the right vertical arrow in the diagram is an injection. This implies the inclusion $\cG_{w/(p-1)+}\subseteq \cC$. Thus we obtain the inclusions
\[
\cC=\cG_{(1-w)/(p-1)} \subseteq \cG_{w/(p-1)+}\subseteq \cC
\]
and the equality $\cC=\cG_b$ holds for any $b$ satisfying $w/(p-1)<b \leq (1-w)/(p-1)$. The assertion for upper ramification subgroups follows from the part (\ref{cansubmain-iso}) and Theorem \ref{TF}.
\end{proof}



\section{Higher canonical subgroups}

In this section, we prove Theorem \ref{main} and Corollary \ref{cansubfamily}. Though this can be done basically by repeating arguments in \cite{AM}, \cite{Fa} and \cite{Ti_HN}, we give a proof here with necessary modifications for the convenience of the reader. First we recall the following two lemmas in \cite{Fa} whose proofs depend only on elementary arguments and are independent of the theory of Harder-Narasimhan filtrations or Hodge-Tate maps developed there. Note that for Lemma \ref{Hasse}, the proof of \cite[Th\'{e}or\`{e}me 5]{Fa} remains valid also for our subgroup scheme $\cC$ by Theorem \ref{cansubmain} (\ref{cansubmain-1}).

\begin{lem}[\cite{Fa}, Th\'{e}or\`{e}me 5]\label{Hasse}
Let $\cG$ be a truncated Barsotti-Tate group of level two, height $h$ and dimension $d<h$ over $\okey$ with $\Hdg(\cG)<1/(p+1)$. Let $\cC$ be the level one canonical subgroup of $\cG[p]$ constructed in Theorem \ref{cansubmain}. Then the group scheme $p^{-1}\cC/\cC$ is a truncated Barsotti-Tate group of level one, height $h$ and dimension $d$ with $\Hdg(p^{-1}\cC/\cC)=p\Hdg(\cG)$.
\end{lem}

\begin{lem}[\cite{Fa}, Proposition 12]\label{freeness}
Let $\cG$ be a truncated Barsotti-Tate group of level $n$ and dimension $d$ over $\okey$. Let $\cC$ be a finite flat closed subgroup scheme of $\cG$ over $\okey$ of order $p^{nd}$. Suppose that we have the inequality $\deg(\cC)>nd-1$. Then the group $\cC(\okbar)$ is isomorphic to $(\bZ/p^n\bZ)^d$.
\end{lem}

\begin{proof}[Proof of Theorem \ref{main}]
We proceed by induction on $n$. The case of $n=1$ is Theorem \ref{cansubmain}. For $n\geq 2$, suppose that the theorem holds for any level less than $n$. By assumption, we have the level one canonical subgroup $\cC$ of $\cG[p]$ as in Theorem \ref{cansubmain}. Then Lemma \ref{Hasse} implies that, for $n\geq 2$, the group scheme $p^{-(n-1)}\cC/\cC$ is a truncated Barsotti-Tate group of level $n-1$ over $\okey$ with $\Hdg(p^{-(n-1)}\cC/\cC)=pw$. Using the induction hypothesis for $p^{-(n-1)}\cC/\cC$, we define the level $n$ canonical subgroup $\cC_n$ of $\cG$ to be the unique finite flat closed subgroup scheme of $\cG$  over $\okey$ containing $\cC$ such that the quotient $\cC_n/\cC$ is the level $n-1$ canonical subgroup of $p^{-(n-1)}\cC/\cC$. Then $\cC_n$ is of order $p^{nd}$.

By the assertion on the Frobenius kernel for $p^{-(n-1)}\cC/\cC$, we have the equality $\cC_n=(F^{n-1})^{-1}(\cC^{(p^{n-1})})$ of fppf sheaves over $\Spec(\okey/m_K^{\geq 1-p^{n-1}w})$. Since the subgroup scheme $\cC\times \Spec(\okey/m_K^{\geq 1-p^{n-1}w})$ also coincides with the Frobenius kernel of $\cG\times \Spec(\okey/m_K^{\geq 1-p^{n-1}w})$, we obtain the assertion on the Frobenius kernel for $\cG$.

Since the multiplication by $p^{n-1}$ induces an isomorphism $\cG/p^{-(n-1)}\cC\to \cG[p]/\cC$, we have the equality 
\[
\deg(\cG/\cC_n)=\deg(\cG[p]/\cC)+\deg((p^{-(n-1)}\cC/\cC)/(\cC_n/\cC))
\]
and the part (\ref{main-deg}) of the theorem follows from the induction hypothesis. 

For the part (\ref{main-iso}), it is enough to show the Cartier pairing $\langle\ ,\ \rangle_\cG$ kills the subset $\cC_n(\okbar)\times \cC_n'(\okbar)$. Let $\cC'$ be the canonical subgroup of $\cG^\vee[p]$. By the construction of $\cC_n'$, we have $\cC'\subseteq \cC_n'$ and $\cC_n'/\cC'\subseteq p^{-(n-1)}\cC'/\cC'$. For $x\in \cC(\okbar)$ and $y\in \cC'_n(\okbar)$, we have $p^{n-1}y \in \cC'(\okbar)$ and $\langle x, y\rangle_\cG=\langle x,p^{n-1}y\rangle_{\cG[p]}=0$ by Theorem \ref{cansubmain} (\ref{cansubmain-iso}). Thus the subgroup $\cC_n'(\okbar)$ is contained in $\cC(\okbar)^\bot=(\cG/\cC)^\vee(\okbar)\subseteq \cG^\vee(\okbar)$ and we have $\langle x, y\rangle_\cG=\langle \bar{x},y\rangle_{\cG/\cC}$ for any $x\in \cC_n(\okbar)$ and $y\in \cC'_n(\okbar)$, where $\bar{x}$ is the image of $x$ in $(\cG/\cC)(\okbar)$. Since we have an exact sequence
\[
0 \to p^{-(n-1)}\cC/\cC\to \cG/\cC\overset{p^{n-1}}{\to} \cG[p]/\cC\to 0
\]
and an isomorphism $\cC'\simeq (\cG[p]/\cC)^\vee$ of subgroup schemes of $\cG^\vee[p]$, a similar argument on the side of $\cC'_n$ implies the equality $\langle \bar{x},y\rangle_{\cG/\cC}=\langle \bar{x},\bar{y}\rangle_{p^{-(n-1)}\cC/\cC}$, where $\bar{y}$ is the image of $y$ by the surjection $(\cG/\cC)^\vee(\okbar) \to (p^{-(n-1)}\cC/\cC)^\vee(\okbar)$. Moreover, we have an isomorphism $(\cG/\cC)^\vee\simeq p^{-(n-1)}\cC'$ of subgroup schemes of $\cG^\vee$ fitting into the commutative diagram
\[
\xymatrix{
0 \ar[r] & (\cG[p]/\cC)^\vee \ar[d]_{\wr} \ar[r] &(\cG/\cC)^\vee \ar[r]\ar[d]_{\wr} &(p^{-(n-1)}\cC/\cC)^\vee \ar[r]\ar@{.>}[d] &0 \\
0 \ar[r] & \cC' \ar[r] & p^{-(n-1)}\cC' \ar[r] & p^{-(n-1)}\cC'/\cC' \ar[r] &0.
}
\]
Thus we obtain an isomorphism $(p^{-(n-1)}\cC/\cC)^\vee\to p^{-(n-1)}\cC'/\cC'$ and the equality $\langle \bar{x},\bar{y}\rangle_{p^{-(n-1)}\cC/\cC}=0$ holds by induction hypothesis.

Moreover, for $w<(p-1)/(p^n-1)$, the inequality $\deg(\cC_n)=nd-w(p^n-1)/(p-1)>nd-1$ holds by the part (\ref{main-deg}) and Lemma \ref{freeness} implies the part (\ref{main-free}).

Next we show the part (\ref{main-sub}). By induction hypothesis, we have the subgroup scheme $\cC_i$ of $\cG[p^i]$ as in the theorem. Note that, by the part (\ref{main-free}), a subgroup of $\cC_n(\okbar)$ which is isomorphic to $(\bZ/p^{n-1}\bZ)^d$ is equal to $\cC_n(\okbar)[p^{n-1}]$. Hence, by induction hypothesis, it is enough to check that $\cC_{n-1}$ is contained in $\cC_n$. The case of $n=2$ follows from the definition of $\cC_n$. Suppose $n>2$. For any $i$ with $2\leq i \leq n$, the group scheme $\cC_{i}/\cC$ is the level $i-1$ canonical subgroup of the truncated Barsotti-Tate group $p^{-(i-1)}\cC/\cC$ of Hodge height $pw$. The inequality $p(p-1)/(p^n-1)<(p-1)/(p^{n-1}-1)$ and the induction hypothesis imply that $\cC_{n-1}/\cC$ is contained in $\cC_n/\cC$. Thus the part (\ref{main-sub}) follows.

Let us show the assertion (\ref{main-HT}). Assume $w<(p-1)/p^n$ and set $\cK_n$ to be the scheme-theoretic closure in $\cG$ of the subgroup $\Ker(\HT_{n-w(p^n-1)/(p-1)})$. We have the inequality $w/(p-1)<1-\epsilon \leq 1-w$ for $\epsilon=w(p^n-1)/(p-1)$ and Theorem \ref{cansubmain} (\ref{cansubmain-HT}) implies that the kernel of the map $\HT_{1-\epsilon}: \cG[p](\okbar) \to \omega_{\cG[p]^\vee}\otimes\okbar/m_{\Kbar}^{\geq 1-\epsilon}$ is of order $p^d$. Using this, we can show that the order of the group scheme $\cK_n$ is no more than $p^{nd}$, by an elementary argument just as in \cite[Proposition 13]{Fa}. On the other hand, we have a commutative diagram with exact rows
\[
\xymatrix{
0 \ar[r] & \cC(\okbar)\ar[d]_{\HT} \ar[r] &\cG(\okbar) \ar[r]\ar[d]_{\HT} &(\cG/\cC)(\okbar) \ar[r]\ar[d]_{\HT} &0 \\
0 \ar[r] &\omega_{\cC^\vee}\otimes \okbar \ar[r] &\omega_{\cG^\vee}\otimes \okbar \ar[r] &\omega_{(\cG/\cC)^\vee}\otimes \okbar \ar[r] &0,
}
\]
where we put $\HT(x)=(x^\vee)^{*}(dt/t)$ as before. Take $x\in \cC_n(\okbar)$ and let $\bar{x}$ denote its image in $(\cG/\cC)(\okbar)$. By the construction of $\cC_n$ and induction hypothesis, we see that the element $\HT(\bar{x})$ is killed by the ideal $m_{\Kbar}^{\geq w(p^n-p)/(p-1)}$. Since $\deg(\cC^\vee)=w$, we have $m_{\Kbar}^{\geq w}(\omega_{\cC^\vee}\otimes \okbar)=0$. Hence the element $\HT(x)$ is killed by the ideal $m_{\Kbar}^{\geq w(p^n-1)/(p-1)}$ and we obtain the inclusion $\cC_n\subseteq \cK_n$, which implies the assertion (\ref{main-HT}) by comparing orders.

Finally, the assertion (\ref{main-ram}) follows just as in the proof of \cite[Theorem 2.5]{Ti_HN}, using induction hypothesis and assertions of the theorem we have already proved. This concludes the proof of Theorem \ref{main}.
\end{proof}

We can also prove the following result on anti-canonical isogenies, slightly generalizing \cite[Proposition 16]{Fa}. 
\begin{prop}\label{noncanisog}
Let $\cG$ be a truncated Barsotti-Tate group over $\okey$ of level two, height $h$, dimension $d$ with $0<d<h$ and Hodge height $w=\Hdg(\cG)$. Suppose $w<1/2$ and let $\cC$ be the canonical subgroup of $\cG[p]$ as in Theorem \ref{cansubmain}. Let $\cD$ be a finite flat closed subgroup scheme of $\cG[p]$ over $\okey$ such that the natural map $\cC(\okbar)\oplus \cD(\okbar)\to \cG[p](\okbar)$ is an isomorphism.
\begin{enumerate}
\item\label{noncan1}
The truncated Barsotti-Tate group $p^{-1}\cD/\cD$ of level one has Hodge height $\Hdg(p^{-1}\cD/\cD)=p^{-1}w$.
\item\label{noncan2}
The subgroup scheme $\cG[p]/\cD$ is the canonical subgroup of $p^{-1}\cD/\cD$.
\item\label{noncan3}
$\deg(\cD)=p^{-1}w$.
\end{enumerate}
\end{prop}
\begin{proof}
By a base change argument as before, we may assume that the residue field $k$ of $K$ is perfect. Then, by \cite[Th\'{e}or\`{e}me 4.4 (e)]{Il}, there exists a Barsotti-Tate group $\Gamma$ over $\okey$ satisfying $\cG\simeq \Gamma[p^2]$.

Note that the truncated Barsotti-Tate group $p^{-1}\cD/\cD$ is also of height $h$ and dimension $d$. The natural homomorphism $\cC\to \cG[p]/\cD$ induces an isomorphism between the generic fibers of both sides. Now we claim that the group scheme $(\cG[p]/\cD)\times \Spec(\okey/m_K^{\geq 1-w})$ is killed by the Frobenius. Indeed, let $\SGn$ and $\SGn'$ be the objects of the category $\ModSGf$ corresponding to the finite flat group schemes $\cC$ and $\cG[p]/\cD$ via the anti-equivalence $\cG(-)$, respectively. By \cite[Corollary 2.2.2]{Li_FC}, the generic isomorphism $(\cG[p]/\cD)^\vee \to \cC^\vee$ corresponds to an injection $\SGn^\vee \to (\SGn')^\vee$. Then the $\SG_1$-modules $\wedge^d \SGn^\vee$ and $\wedge^d(\SGn')^\vee$ are free of rank one and we also have an injection $\wedge^d \SGn^\vee\to \wedge^d(\SGn')^\vee$. Hence we obtain the inequality $v_u(\det \varphi_{(\SGn')^\vee}) \leq v_u(\det \varphi_{\SGn^\vee})$. By Theorem \ref{cansubmain} (\ref{cansubmain-iso}), the object $\SGn^\vee$ corresponds to the quotient $\cG[p]^\vee/\cC'$ of the Cartier dual $\cG[p]^\vee$ by its canonical subgroup $\cC'$. Thus Lemma \ref{invdiff} (\ref{invdiff-deg}) and Theorem \ref{cansubmain} (\ref{cansubmain-deg}) imply the equality $v_u(\det \varphi_{\SGn^\vee})=ew$. From this and the definition of the dual object of the category $\ModSGf$, we see that the map $\varphi_{\SGn'}$ is zero modulo $u^{e(1-w)}$ and the claim follows from Theorem \ref{ramcorr} (\ref{ramcorr2}). 

Since the group scheme $p^{-1}\cD/\cD$ is a truncated Barsotti-Tate group of level one and dimension $d$, the group scheme in the claim coincides with the kernel of the Frobenius of the group scheme $(p^{-1}\cD/\cD)\times \Spec(\okey/m_K^{\geq 1-w})$. Note the equality $p^{-1}(\cG[p]/\cD)=\cG/\cD$ of closed subgroup schemes of the truncated Barsotti-Tate group $p^{-2}\cD/\cD$. Therefore the Frobenius induces an isomorphism of group schemes over $\Spec(\okey/m_K^{\geq 1-w})$
\[
\cG[p]\simeq (\cG/\cD)/(\cG[p]/\cD)=p^{-1}(\cG[p]/\cD)/(\cG[p]/\cD)\overset{F}{\to} (p^{-1}\cD/\cD)^{(p)}.
\]
Considering the Hodge heights of both sides, we have the equality
\[
\min \{w, 1-w\}=\min \{p\Hdg(p^{-1}\cD/\cD),1-w\},
\]
from which the assertion (\ref{noncan1}) follows. The uniqueness of the canonical subgroup in Theorem \ref{cansubmain} (\ref{cansubmain-1}) implies the assertion (\ref{noncan2}). The last assertion follows from the isomorphism $(p^{-1}\cD/\cD)/(\cG[p]/\cD)\simeq \cD$ and Theorem \ref{cansubmain} (\ref{cansubmain-deg}).
\end{proof}

We give here remarks on the uniqueness of the canonical subgroup $\cC_n$. Let $\cG$ be a truncated Barsotti-Tate group over $\okey$ of level $n$, height $h$, dimension $d$ and Hodge height $w<1/(p^{n-2}(p+1))$. Let $\cD_n$ be a finite flat closed subgroup scheme of $\cG$ over $\okey$. Then an induction and the uniqueness assertion in Theorem \ref{cansubmain} (\ref{cansubmain-1}) show the following uniqueness of $\cC_n$: Suppose that there exists a filtration $0=\cD_0\subseteq \cD_1\subseteq \cD_2\subseteq\cdots \subseteq \cD_{n-1}\subseteq \cD_n$ of finite flat closed subgroup schemes over $\okey$ such that $\cD_i/\cD_{i-1}$ is killed by $p$, of order $p^d$ and its modulo $m_K^{\geq 1-p^{i-1}w}$ is killed by the Frobenius for any $i$. Then we have $\cC_n=\cD_n$. 

On the other hand, the following stronger uniqueness also holds for $w<p(p-1)/(p^{n+1}-1)$. Note the inequalities 
\[
(p-1)/p^n<p(p-1)/(p^{n+1}-1)< (p-1)/(p^n-1).
\]

\begin{prop}\label{uniqueFrn}
Let $\cG$ be a truncated Barsotti-Tate group over $\okey$ of level $n$, height $h$, dimension $d$ and Hodge height $w<p(p-1)/(p^{n+1}-1)$. Let $\cD_n$ be a finite flat closed subgroup scheme of $\cG$ over $\okey$ such that $\cD_n(\okbar)\simeq (\bZ/p^n\bZ)^d$ and the group scheme $\cD_n\times\Spec(\okey/m_K^{\geq 1-p^{n-1}w})$ is killed by the $n$-th iterated Frobenius $F^n$. Then we have $\cC_n=\cD_n$.
\end{prop}
\begin{proof}
We proceed by induction on $n$. The case of $n=1$ follows from Theorem \ref{cansubmain} (\ref{cansubmain-1}). Suppose that $n\geq 2$ and the assertion holds for $n-1$. Let $\cC_1$ be the scheme-theoretic closure of $\cC_n(\okbar)[p]$ in $\cC_n$ and define $\cD_1$ similarly. By Theorem \ref{main} (\ref{main-sub}), the subgroup scheme $\cC_1$ is the canonical subgroup of $\cG[p]$. First we claim $\cC_1=\cD_1$. For this, by a base change argument as before, we may assume that the residue field $k$ of $K$ is perfect and $\cG(\okbar)=\cG(\okey)$. Suppose $\cC_1\neq \cD_1$. Then we can find a finite flat closed subgroup scheme $\cE$ of $\cG[p]$ such that $\cC_1(\okbar)\oplus \cE(\okbar)=\cG[p](\okbar)$ and $\cD_1(\okbar)\cap \cE(\okbar)\neq 0$. Let $\cF$ be the scheme-theoretic closure of the latter intersection in $\cG[p]$, which is a closed subgroup scheme of $\cE$. Let $\SGl$, $\SGe$ and $\SGf$ be the objects of the category $\ModSGf$ corresponding to $\cG[p]/\cC_1$, $\cE$ and $\cF$, respectively. Since we have the inequality $w<p(p-1)/(p^3-1)<1/2$, Proposition \ref{noncanisog} (\ref{noncan3}) implies $v_u(\det \varphi_\SGf)\leq  v_u(\det\varphi_{\SGe})=p^{-1}ew$. On the other hand, since $\cF$ is also a closed subgroup scheme of $\cD_n$, the $n$-th iterated Frobenius of $\cF\times\Spec(\okey/m_K^{\geq 1-p^{n-1}w})$ is zero. By Theorem \ref{ramcorr} (\ref{ramcorr2}), we see that the entries of a representing matrix of the map
\[
\SG_1\otimes_{\varphi^n,\SG_1}\SGf \overset{1\otimes \varphi_\SGf}{\to} \SG_1\otimes_{\varphi^{n-1},\SG_1}\SGf\overset{1\otimes \varphi_\SGf}{\to}\cdots\overset{1\otimes \varphi_\SGf}{\to}\SG_1\otimes_{\varphi,\SG_1}\SGf\overset{1\otimes \varphi_\SGf}{\to}\SGf
\]
have $u$-adic valuations no less than $e(1-p^{n-1}w)$. Taking the valuation of the determinant of this map, we obtain the inequalities 
\[
1-p^{n-1}w\leq e^{-1}v_u(\det\varphi_\SGf)(p^n-1)/(p-1)\leq w(p^n-1)/(p^2-p), 
\]
which contradict the assumption on $w$ and the equality $\cC_1=\cD_1$ follows.

Now the group scheme $p^{-(n-1)}\cC_1/\cC_1$ is a truncated Barsotti-Tate group of level $n-1$, height $h$, dimension $d$ and Hodge height $pw$ with the level $n-1$ canonical subgroup $\cC_n/\cC_1$. The finite flat closed subgroup scheme $\cD_n/\cC_1$ satisfies $(\cD_n/\cC_1)(\okbar)\simeq (\bZ/p^{n-1}\bZ)^d$ and its modulo $m_K^{\geq 1-p^{n-2}pw}$ is killed by $F^{n-1}$. Thus the induction hypothesis implies the equality $\cC_n/\cC_1=\cD_n/\cC_1$. This concludes the proof of the proposition.
\end{proof}

To show Corollary \ref{cansubfamily}, we need a patching lemma for upper ramification subgroups, which is a slight generalization of \cite[Proposition 8.2.2]{AM}. Let $K/\bQ_p$ be an extension of complete discrete valuation fields and $j$ be a positive rational number. Let $\frX$ be an admissible formal scheme over $\Spf(\okey)$ which is quasi-compact and $X$ be its Raynaud generic fiber. For a finite extension $L/K$, an $\oel$-valued point $x:\Spf(\oel)\to \frX$ and an ideal $\sI$ of $\cO_{\frX}$, we let $\sI(x)$ denote the ideal of $\oel$ generated by the image of $\sI$. The $p$-adic valuation of a generator of this ideal is denoted by $v_p(\sI(x))$. For any $x\in X(L)$, we let $x$ also denote the map $\Spf(\oel)\to \frX$ obtained from $x$ by taking the scheme-theoretic closure and the normalization.

\begin{lem}\label{familyup}
Let the notation be as above. Let $\frG$ be a finite locally free formal group scheme over $\frX$, $G$ be its Raynaud generic fiber and $\sH$ be a coherent open ideal of $\cO_{\frX}$.
Then there exists an admissible open subgroup $G^{j\sH+}$ of $G$ such that for any finite extension $L/K$ and $x\in X(L)$, the fiber $G_x^{j\sH+}$ coincides with the generic fiber of the upper ramification subgroup scheme $\frG_{x}^{jv_p(\sH(x))+}$ of the finite flat (formal) group scheme $\frG_{x}=\frG\times_{\frX,x}\Spf(\oel)$ over $\oel$. 
\end{lem}
\begin{proof}
By replacing $\frX$ by the admissible blow-up of $\frX$ along the ideal $\cH$, we may assume that the ideal $\cH$ is invertible. It suffices to show the existence of an admissible open subgroup as in the lemma locally. Thus we may assume that $\frX=\Spf(B)$ is affine, the ideal $\cH$ is generated by $h\in B$ and there exists a closed immersion of $\frG$ to a projective abelian scheme $\cA$ over $\Spec(B)$, by a theorem of Raynaud (\cite[Th\'{e}or\`{e}me 3.1.1]{BBM}). Let $\frP$ be the formal completion of $\cA$ along the special fiber, $P$ be its Raynaud generic fiber and $\sI$ be the defining ideal of the closed immersion $\frG\to \frP$. For positive rational numbers $j$ and $j'$, write $j=m/n$ and $j'=m'/n$ with positive integers $m,m',n$ and put $\sJ=\sI^n+p^{m'}h^m\cO_\frP$. Let $\frB$ be the admissible blow-up of $\frP$ along the ideal $\sJ$ and $\frZ^{m,m',n,h}$ be the formal open subscheme of $\frB$ where $p^{m'}h^m$ generates the ideal $\sJ\cO_\frB$. 

The Raynaud generic fiber of $\frZ^{m,m',n,h}$ is the admissible open subset of $P$ whose set of $\Kbar$-valued points is given by
\[
\{y\in P(\Kbar)| v_p(\sI(y))\geq j v_p(h(y)) +j'\}
\]
and it is independent of the choice of $m,m',n$. We write this Raynaud generic fiber as $Z^{jh+j'}$. We claim that the admissible open subset $Z^{jh+j'}$ is a rigid-analytic subgroup of $P$. For this, take a finite extension $L/K$ and $x\in X(L)$. It is enough to show that the fiber $(Z^{jh+j'})_x$ is a rigid-analytic subgroup of $P_x$. This is the admissible open subset of $P_x$ consisting of $y\in P_x(\bar{L})$ such that the inequality $v_L(\sI(y))\geq e(L/\bQ_p)(jv_p(h(x))+j')$ holds, where $v_L$ is the normalized valuation of $L$. Taking a sufficiently large $L$, we may assume that the constant $e(L/\bQ_p)(jv_p(h(x))+j')$ is an integer. Then the universality of the dilatation implies the claim, as in the proof of \cite[Proposition 8.2.2]{AM}.

By \cite[Proposition A.1.2]{AM}, there exists an admissible open subgroup $Z^{jh+j',0}$ of $Z^{jh+j'}$ such that for any $x\in X(\Kbar)$, the fiber $(Z^{jh+j',0})_x$ coincides with the unit component $(Z^{jh+j'}_x)^0$ of the rigid-analytic group $Z^{jh+j'}_x$. Put 
\[
Z^{jh+}=\bigcup_{j'>0}Z^{jh+j'}, Z^{jh+,0}=\bigcup_{j'>0}Z^{jh+j',0}
\]
and $G^{j\sH+}=G\cap Z^{jh+,0}$. These are admissible open subgroups of $P$ and $G$, respectively. To show that $G^{j\sH+}$ satisfies the property as in the lemma, it is enough to show that, for any finite extension $L/K$ and $x\in X(L)$, the fiber $(Z^{jh+j'})_x$ is naturally isomorphic to the tubular neighborhood $X^ {jv_p(h(x))+j'}(\frG_x\to \frZ)$ of a closed immersion from $\frG_x$ to a formal affine scheme $\frZ$ which is an object of the category $\mathrm{SFA}_{\oel}$ (\cite[Notation 1.5]{AM}). Take $x\in X(L)$. Since $\cA$ is projective and $\frG_x$ is finite over $\oel$, there exists an affine open subscheme $\cU_x$ of $\cA_x$ such that the closed immersion $\frG_x\to \cA_x$ factors through $\cU_x$. Let $\hat{\cU}_x$ be its formal completion along the special fiber and $U_x$ be its Raynaud generic fiber. Then $\hat{\cU}_x$ is an object of the category $\mathrm{SFA}_{\oel}$. Moreover, $\hat{\cU}_x$ is a formal affine open subscheme of $\frP_x$ through which the closed immersion $\frG_x \to \frP_x$ factors and the restriction $(Z^{jh+j'})_x|_{U_x}$ coincides with the tubular neighborhood $X^{jv_p(h(x))+j'}(\frG_x\to \hat{\cU}_x)$. Therefore we are reduced to show the equality $(Z^{jh+j'})_x|_{U_x}=(Z^{jh+j'})_x$. For this, take $y\in (Z^{jh+j'})_x(\bar{L})$. Then we have $v_p(\sI(y))\geq jv_p(h(x))+j'> 0$. Put $\frP_x\setminus \hat{\cU}_x=V(\sI'_x)$ with some ideal $\sI'_x\subseteq \cO_{\frP_x}$. Let $\bar{\frP}_x$ and $\bar{y}$ denote the special fibers of $\frP_x$ and $y$, respectively. Consider the inverse image $\bar{\sI}_x$ ({\it resp.} $\bar{\sI}'_x$) of $\sI_x$ ({\it resp.} $\sI'_x$) in $\bar{\frP}_x$. Since $V(\bar{\sI}_x)\cap V(\bar{\sI}'_x)=\emptyset$, we have $\bar{\sI}_x + \bar{\sI}'_x=(1)$ in any open affine neighborhood of the image of $\bar{y}$ in $\bar{\frP}_x$. Thus the image is not contained in $V(\bar{\sI}'_x)$ and we obtain $y\in (Z^{jh+j'})_x|_{U_x}$. 
\end{proof}

\begin{proof}[Proof of Corollary \ref{cansubfamily}]
Define a positive rational number $j$ and an open coherent ideal $\sH$ of $\cO_{\frX}$ by $j=p(p^n-1)/(p-1)^2$ and $\sH=\Ha_{\frI}(\frG[p])$ (\cite[2.2.2]{Fa}). We define $C_n$ to be the admissible open subgroup $G^{j\sH+}$ of $G$ as in Lemma \ref{familyup}. Then, by Theorem \ref{main} (\ref{main-ram1}) and (\ref{main-ram}), each fiber $(C_n)_x$ coincides with the generic fiber of the level $n$ canonical subgroup of $\frG_x$. Since $G$ is etale over $X$, the admissible open subgroup $C_n$ of $G$ is also etale over $X$. Since its fibers over $X(r_n)$ are isomorphic to the group $(\bZ/p^n\bZ)^d$ by Theorem \ref{main} (\ref{main-free}), \cite[Lemme A.1.1]{AM} implies that $C_n$ is finite over $X(r_n)$. Hence the rigid-analytic group $C_n$ over $X(r_n)$ is etale locally constant, and thus etale locally on $X(r_n)$ this is isomorphic to the constant group $(\bZ/p^n\bZ)^d$.
\end{proof}



\end{document}